\documentclass{amsart}
\usepackage{amssymb, amscd, ascmac}
\usepackage{pb-xy,pb-diagram}
\numberwithin{equation}{section}
\newtheorem{theorem}{Theorem}[section]
\newtheorem{proposition}[theorem]{Proposition}
\newtheorem{lemma}[theorem]{Lemma}

\newtheorem{conjecture}[theorem]{Conjecture}
\theoremstyle{definition}

\newtheorem{remark}[theorem]{Remark}

\title[The cycle map]{On the cycle map of a finite group}
\author{Masaki Kameko}
\address{
Department of Mathematical Sciences,
Shibaura Institute of Technology,
307 Minuma-ku Fukasaku, Saitama-City 337-8570, Japan}
\email{kameko@shibaura-it.ac.jp}
\begin{document} 
\begin{abstract}
Let $p$ be an odd prime number.
We show that 
there exists a finite group of order $p^{p+3}$ 
whose the mod $p$ cycle map from the mod $p$ Chow ring of 
its classifying space to its ordinary  mod $p$ cohomology
is not injective.
\end{abstract}
\maketitle
\section{Introduction}\label{section:1}

The Chow group $CH^iX$ of a smooth 
algebraic variety $X$ is the group of finite $\mathbb Z$-linear combinations 
of closed subvarieties of $X$ of codimension $i$ modulo rational equivalence
and
$\bigoplus_{i\geq 0}  CH^iX$, called the Chow ring of $X$, is a
ring under intersection product.
It is an important subject to study in algebraic geometry. 
For a smooth complex algebraic variety,
the cycle map is a homomorphism from the Chow ring to 
 the ordinary integral cohomology of the underlying topological space.
Thus,  
the cycle map relates algebraic geometry to algebraic topology.
In \cite{totaro-1999}, Totaro 
considered the Chow ring of the classifying space $BG$ of an algebraic group $G$. 
In his recently published book \cite{totaro-2014}, for each prime number $p$, 
Totaro gave an example of finite 
group $K$ of order $p^{2p+1}$ such that the mod $p$ cycle map
\[
cl:CH^2BK/p \to H^4(BK)
\]
is not injective, where $H^*(-)$ is the ordinary mod $p$ cohomology
and  the finite group $K$ is regarded as a complex  algebraic group.
Totaro  wrote 
\lq\lq ... but there are probably smaller examples\rq\rq \,
in his book.

In this paper, we find a smaller example, possibly the smallest one.
To be precise, we construct a finite group $H$ of order $p^{p+3}$ 
to prove the following result.

\begin{theorem}\label{theorem:1.1}
For each prime number $p$, there exists a finite group $H$ 
of order $p^{p+3}$ such that
the mod $p$ cycle map
$
cl: CH^2BH/p \to H^4(BH)
$
 is not injective, where the finite group $H$ is regarded as a complex algebraic group.
\end{theorem}

For a complex algebraic group $G$, the following results were obtained 
by Totaro using Merkurjev's theorem  in \cite[Corollary 3.5]{totaro-1999}.
\begin{enumerate}
\item $CH^2BG$ is generated by Chern classes.
\item $CH^2BG\to H^{4}(BG; \mathbb{Z})$ is injective.
\end{enumerate}
Thus, we may use the ordinary integral cohomology 
and Chern classes to study the Chow group
$CH^2BG$. A problem concerning 
the Chow group $CH^2BG$ in algebraic geometry could be viewed
as a problem on the Chern subgroup of the ordinary integral cohomology $H^4(BG;\mathbb{Z})$, 
that is, the subgroup of $H^4(BG;\mathbb{Z})$ generated by 
Chern classes of complex representations of
$G$, in classical algebraic topology.
In what follows, we consider $CH^2(BG)$ as the Chern subgroup 
of the integral cohomology $H^4(BG; \mathbb{Z})$ 
and the mod $p$ cycle map $CH^2BG/p \to H^4(BG)$ 
as the homomorphism induced by the mod $p$ reduction 
$\rho:H^4(BG; \mathbb{Z}) \to H^4(BG)$.
Since we consider the ordinary integral and mod $p$ cohomology only, 
the group $G$ could be a topological group and 
it need not to be a complex algebraic group.

Throughout the rest of this paper, we assume that $p$ is an odd prime number
unless otherwise stated explicitly.
Let $p_{+}^{1+2}$ be the extraspecial $p$-group of order $p^3$ with exponent $p$. 
We consider it as a subgroup of the special unitary group $SU(p)$.
We will define a subgroup $H_2$ of $SU(p)$ in Section~\ref{section:2}.
The group $H$ in Theorem~\ref{theorem:1.1} is given in terms of $p_{+}^{1+2}$ 
and $H_2$, that is 
\[
H=p_{+}^{1+2} \times H_2/\langle \Delta(\xi)\rangle, 
\]
where
$\langle \Delta(\xi)\rangle$ is a cyclic group in the center of $SU(p)\times SU(p)$.
We define the group $G$ as 
\[
G=SU(p)\times SU(p)/\langle \Delta(\xi)\rangle.
\]
We will give the detail of $G$, $H$ and $H_2$  in Section~\ref{section:2}.
What we prove in this paper is the following theorem.

\begin{theorem}
\label{theorem:1.2}
Let $p$ be an odd prime number.
Let $K$ be a subgroup of 
\[
G=SU(p)\times SU(p) /\langle \Delta(\xi) \rangle
\]
containing 
\[
H=p_{+}^{1+2} \times H_2 /\langle \Delta(\xi) \rangle.
\]
Then, the mod $p$ cycle map 
$
cl: CH^2BK/p \to H^4(BK)
$
is  not injective.
\end{theorem}

The order of the group $p_{+}^{1+2} \times H_2/\langle\Delta(\xi)\rangle$ is 
$p^{p+3}$ and it is the group $H$ in 
Theorem~\ref{theorem:1.1}.
Applying Theorem~\ref{theorem:1.2} to
\[
K=p_{+}^{1+2}\times \left( (\mathbb{Z}/p^2)^{p-1} \rtimes 
\mathbb{Z}/p\right) /\langle\Delta(\xi)\rangle, 
\]
we obtain Totaro's example in \cite[Section 15]{totaro-2014}.
Thus our result not only gives a smaller group whose mod $p$ cycle map is 
not injective but it extends Totaro's result.
For $p=2$,  Theorem~\ref{theorem:1.1} was proved by Totaro in 
\cite[Theorem 15.13]{totaro-2014}.
For $p=2$, the finite group $H$ is the extraspecial $2$-group 
$2_{+}^{1+4}$ of order $2^5$.
It is not difficult to see 
that we cannot replace $H_2$ by the extraspecial $p$-group $p_{+}^{1+2}$
in Theorem~\ref{theorem:1.2}. See Remark~\ref{remark:6.3}.
This observation leads us to the following conjecture:

\begin{conjecture}\label{conjecture:1.3}
Let $p$ be a prime number. 
For finite $p$-group $K$ of order less than $p^{p+3}$,
the mod $p$ cycle map
$
cl:CH^2BK/p\to H^4(BK)
$
is injective.
\end{conjecture}

This paper is organized as follows:
In Section~\ref{section:2}, we define groups 
that we use in this paper including $G$, $H$ above.
In Section~\ref{section:3}, we recall the cohomology of the classifying space of 
the projective unitary group $PU(p)$ up to degree $5$.
In Section~\ref{section:3}, we prove that the mod $p$ cycle map 
$CH^2BG/p\to H^4(BG)$ is not injective and describe its kernel.
In Section~\ref{section:4}, we collect some properties of the mod $p$ cohomology of 
$B\tilde{\pi}(H_2)$, where $\tilde{\pi}$ is the restriction of the projection from $SU(p)$ to $PU(p)$.
We use the mod $p$  cohomology of $B\tilde{\pi}(H_2)$ in Section~\ref{section:5}. 
In Section~\ref{section:5}, we study the mod $p$ cycle map $CH^2BH/p\to H^4(BH)$ 
 to complete the proof of Theorem~\ref{theorem:1.2}.

Throughout the rest of this paper, by abuse of notation, 
we denote the map between classifying spaces
 induced by a group homomorphism $f:G\to G'$ by $f:BG \to BG'$.

The author would like to thank the referee not only for pointing our several errors, 
kind advice and helpful comments
but also for his/her patience.

\section{Subgroups and quotient groups}\label{section:2} 

In this section, we define subgroups of the unitary group 
$U(p)$ and of the product $SU(p) \times SU(p)$ of special unitary group  $SU(p)$.
We also define their quotient groups. 
For a finite subset $\{ x_1, \dots, x_r\}$ of a group, we denote by $\langle 
x_1, \dots, x_r\rangle$ the subgroup generated by $\{x_1, \dots, x_r\}$.
As we already mentioned, we assume that $p$ is an odd prime number.

We start with subgroups of the special unitary group $SU(p)$.
Let $\xi =\exp(2\pi i/p)$, $\omega=\exp(2\pi i/p^2)$
and
$\delta_{ij}=1$ if $i\equiv j \mod p$, $\delta_{ij}=0$ if $i \not \equiv j \mod p$.
We consider the following matrices in $SU(p)$.
\begin{align*}
\xi 
& 
=(\xi \delta_{ij})=\mathrm{diag}(\xi, \dots, \xi),
\\
\alpha 
&
=(\xi^{i-1} \delta_{ij}) =\mathrm{diag}(1, \xi, \dots, \xi^{p-1}),
\\
\beta
&
=(\delta_{i,j-1}), 
\\
\sigma_1
&
= \mathrm{diag}(\omega \xi^{p-1}, \omega, \dots, \omega).
\end{align*}
Moreover, let $\sigma_k$ be the diagonal matrix 
whose $(i,i)$-entry is $\omega \xi^{p-1}$ for $i=k$, 
$\omega$ for $i\not=k$.
Let us consider the following subgroups of $SU(p)$:
\begin{align*}
p_{+}^{1+2} & = \langle \alpha, \beta, \xi \rangle, \\
H_2&=\langle \beta, \sigma_1, \dots, \sigma_p \rangle.
\end{align*}
The group $p_{+}^{1+2}$ is the extraspecial $p$-group 
of order $p^3$ with exponent $p$.
Since $\sigma_1^p=\cdots = \sigma_p^p=\xi$ and since
\[
\sigma_2\sigma_3^2\cdots \sigma_{p}^{p-1}=\xi^{(p-1)/2} \alpha^{-1},
\]
the group $H_2$ contains $p_{+}^{1+2}$ as a subgroup.
An element in the subgroup of $H_2$ generated by $\sigma_1, \dots, \sigma_p$
could be described as 
\[
\omega^j \mathrm{diag}( \xi^{i_1}, \dots, \xi^{i_p}),
\]
 where
$0\leq j \leq p-1$,
$0\leq i_1 \leq p-1$, \dots, 
$0\leq i_p\leq p-1$ and $i_1+\cdots+i_p$ is divisible by $p$. So, the order of this subgroup is $p^{p}$.
Since $\beta$ acts on the subgroup of diagonal matrices as a cyclic permutation, 
 the order of $H_2$ is $p^{p+1}$. 

We write $A_2$ for the quotient group $p_{+}^{1+2}/\langle \xi \rangle$.
The group $A_2$ is an elementary abelian $p$-group of rank $2$.
We denote by $\tilde{\pi}$ the obvious projection $SU(p)\to PU(p)$ and projections induced by 
this projection, e.g, $\tilde{\pi}:p_{+}^{1+2}\to \tilde{\pi}(p_{+}^{1+2})=A_2$.
We denote the obvious inclusions among $p_{+}^{1+2}$, $H_2$ and $SU(p)$
and among $A_2$, $\tilde{\pi}(H_2)$, $PU(p)$ 
by $\iota$.

Let us consider the following maps:
\begin{align*}
\Delta:SU(p) \to SU(p)\times SU(p), \quad m \mapsto 
\begin{pmatrix} m & 0 \\ 0 & m\end{pmatrix}.
\\
\Gamma_1:SU(p) \to SU(p)\times SU(p), \quad m \mapsto 
\begin{pmatrix}m  & 0 \\ 0 & I \end{pmatrix}.
\\
\Gamma_2:SU(p) \to SU(p)\times SU(p), \quad m \mapsto 
\begin{pmatrix} I & 0 \\ 0 & m \end{pmatrix}.
\end{align*}
Using these maps and matrices in $SU(p)$ above, 
we consider the following groups.
\begin{align*}
G
& 
=
SU(p)\times SU(p)/\langle \Delta(\xi) \rangle, 
\\
H
&
=
\langle 
\Delta(\alpha),
 \Delta(\beta), 
 \Delta(\xi), \Gamma_2(\beta),
\Gamma_2(\sigma_1), \dots, 
\Gamma_2(\sigma_p)
\rangle
/
\langle
\Delta(\xi) \rangle, 
\\
A_3
&
=
\langle 
\Delta(\alpha),
 \Delta(\beta), 
 \Delta(\xi), 
\Gamma_2(\xi)
\rangle
/
\langle
\Delta(\xi) \rangle, 
\\
A_3'
&
=
\langle 
\Gamma_1(\alpha),
\Gamma_2(\beta), 
 \Delta(\xi), 
\Gamma_2(\xi)
\rangle
/
\langle
\Delta(\xi) \rangle.
\end{align*}
Since $\alpha$ and $\beta$ are in $H_2$, the subgroup 
\[
\langle 
\Delta(\alpha),
 \Delta(\beta), 
 \Delta(\xi), \Gamma_2(\beta),
\Gamma_2(\sigma_1), \dots, 
\Gamma_2(\sigma_p)
\rangle
\]
 contains 
\[
\Gamma_1(\alpha)=\Delta(\alpha)\Gamma_2(\alpha^{-1}), 
\quad
\Gamma_1(\beta)=\Delta(\beta)\Gamma_2(\beta^{-1}),  
\quad
\Gamma_1(\xi)=\Delta(\xi)\Gamma_2(\xi^{-1}).
\]
Therefore, it is equal to the subgroup 
\[
p_{+}^{1+2} \times H_2=
\langle 
\Gamma_1(\alpha),
 \Gamma_1(\beta), 
 \Gamma_1(\xi), \Gamma_2(\beta),
\Gamma_2(\sigma_1), \dots, 
\Gamma_2(\sigma_p)
\rangle.
\]
Hence, we have
\[
H=p_{+}^{1+2} \times H_2/\langle \Delta(\xi)\rangle.
\]
We denote the obvious inclusion of $H$ by $f:H\to G$.
It is also clear that $A_3$, $A_3'$ are elementary abelian $p$-subgroups of rank $3$.
We use the elementary abelian $p$-subgroup $A_3'$ only in the proof of Proposition~\ref{proposition:6.4}. 
In the above groups, $\Gamma_1(\xi)=\Gamma_2(\xi)$.
We denote by $\pi$ the obvious projections induced by $\pi:G\to PU(p)\times PU(p)$.
It is clear that 
\[
\pi(H)=H/\langle \Gamma_2(\xi)\rangle = A_2\times \tilde{\pi}(H_2)
\] 
and 
\[
PU(p)\times PU(p)=SU(p)\times SU(p)/\langle \Delta(\xi), \Gamma_2(\xi) \rangle.
\]
Moreover, we have the following commutative diagram:
\[
\begin{diagram}
\node{A_3} \arrow{e,t}{g}  \arrow{s,l}{\varphi} 
\node{H}  \arrow{s,r}{\pi}
\node{A_3'} \arrow{s,r}{\varphi'}\arrow{w,t}{g'}
\\
\node{A_2} \arrow{e,t}{g}
\node{A_2\times \tilde{\pi}(H_2)} 
\node{A_2.}\arrow{w,t}{g'}
\end{diagram}
\]
where  upper $g, g'$ are obvious inclusions, 
$A_2=\langle \tilde{\pi}(\alpha), \tilde{\pi}(\beta) \rangle$, 
\[
\begin{array}{ll}
\varphi(\Delta(\alpha))=\tilde{\pi}(\alpha), & \varphi(\Delta(\beta))=\tilde{\pi}(\beta), 
\\[.2cm]
\varphi'(\Gamma_1(\alpha))=\tilde{\pi}(\alpha), & \varphi'(\Gamma_2(\beta))=\tilde{\pi}(\beta), 
\\[.2cm]
g(\tilde{\pi}(\alpha))=(\tilde{\pi}(\alpha), \tilde{\pi}(\alpha)), &g(\tilde{\pi}
(\beta))=(\tilde{\pi}(\beta), \tilde{\pi}(\beta)), 
\\[.2cm]
g'(\tilde{\pi}(\alpha))=(\tilde{\pi}(\alpha), 1), & g'(\tilde{\pi}(\beta))=(1, \tilde{\pi}(\beta)).
\end{array}
\]

We end this section by considering another subgroup 
$H_2'$ of the unitary group $U(p)$ and its quotient group 
$\tilde{\pi}(H_2')$, which is a subgroup of $PU(p)$.
We use $H_2'$ and $\tilde{\pi}(H_2')$  only in the proof of 
Proposition~\ref{proposition:3.4}.
Let
$T^{p}$ be the set of all diagonal matrices in $U(p)$, 
which is a maximal torus of $U(p)$.
We define $H_2'=T^{p} \rtimes \mathbb{Z}/p$ as the subgroup generated 
by $T^{p}$ and $\beta$.
It is clear that $\tilde{\pi}(H_2)$ is a subgroup of $\tilde{\pi}'(H_2') 
\subset PU(p)$, where we denote by $\tilde{\pi}'$ the obvious projection $U(p) \to PU(p)$.

\section{The cohomology of $BPU(p)$}\label{section:3}

In this section, we recall the integral and mod $p$ cohomology of $BPU(p)$. 
Throughout  the rest of this paper, 
we denote the integral cohomology of a space $X$ by $H^{*}(X;\mathbb{Z})$ and its mod $p$ cohomology by $H^{*}(X)$. Also, we denote the mod $p$ reduction 
by \[
\rho:H^{*}(X;\mathbb{Z})\to H^{*}(X).
\]
We also define generators 
$u_2\in H^2(BPU(p))$ and $z_1 \in H^1(B\langle \xi \rangle)$, so that
$d_2(z_1)=x_1y_1$, $d_2(z_1)=u_2$ and $\iota^*(u_2)=x_1y_1$, where
$x_1, y_1 \in H^1(BA_2)$ are generators corresponding to 
$\alpha, \beta$ in $\pi_1(BA_2)=\langle \tilde{\pi}(\alpha), \tilde{\pi}(\beta)\rangle$, $d_2$'s are differentials 
in the Leray-Serre spectral sequence associated with the vertical fibrations 
$\tilde{\pi}$'s in
\begin{equation}
\label{equation:3.1}
\begin{diagram}
\node{Bp_{+}^{1+2}} \arrow{e,t}{\iota} \arrow{s,l}{\tilde{\pi}} 
\node{BSU(p)} \arrow{s,r}{\tilde{\pi}}
\\
\node{BA_2} \arrow{e,t}{\iota} 
\node{BPU(p),}
\end{diagram}
\end{equation}
where vertical maps are induced by the obvious projections and 
horizontal maps are induced by the obvious inclusions.

First, we set up notations related with the spectral sequence.
Let 
\[
\pi:X \to B
\]
a fibration.
Since the base space $B$ is usually clear from the context, 
we write $E_r^{s,t}(X)$ for the Leray-Serre spectral sequence 
associated with the above fibration converging to
the mod $p$ cohomology $H^*(X)$.
If it is clear from the context, we write $E_r^{s,t}$ for the Leray-Serre spectral sequence.
We denote by 
\[
H^{s+t} (X) =F^0 H^{s+t}(X) \supseteq F^1 H^{s+t}(X) \supseteq \cdots 
\supseteq F^{s+t}H^{s+t}(X) =\{0\}
\] the filtration on $H^{s+t}(X)$ 
associated with the spectral sequence.
Unless otherwise stated explicitly, by abuse of notation, we denote
the cohomology class and the element it represents in the spectral sequence 
by the same symbol.
Usually, it is clear from the context whether we deal with the cohomology class 
or the element in
the spectral sequence.
Let $R$ be an algebra or a graded algebra. Let $\{ x_1, \dots, x_r\}$ be a finite set. 
We denote by $R\{x_1, \dots, x_r\}$ the free $R$-module spanned by $\{ x_1, \dots, x_r\}$.
For a graded module $M$, we say $M$ is a free $R$-module up to degree $m$ if 
the $R$-module homomorphism 
\[
f:(R\{ x_1, \dots, x_r\})^i \to M^i
\] is  an isomorphism for $i\leq m$ where 
$\{ x_1, \dots, x_r\}$ is a finite subset of $M$.
We say a spectral sequence collapses at the 
$E_r$-level up to degree $m$ if  $E_r^{s,t}=E_\infty^{s,t}$ for $s+t\leq m$.

Next, we recall the integral and mod $p$ cohomology of $BPU(p)$. 
The mod $3$ cohomology of $BPU(3)$ was computed 
by Kono, Mimura and Shimada in \cite{kono-mimura-shimada-1975}.
The integral and mod $p$ cohomology of $BPU(p)$ was computed 
by Vistoli in \cite {vistoli-2007}.
The mod $p$ cohomology was computed by the author and Yagita in 
\cite{kameko-yagita-2008} independently.
The computation up to degree $5$ was also done 
by Antieau and Williams in \cite{antieau-williams-2014}.
Although the direct computation is not difficult, 
we prove the following proposition by direct computation 
because it is slightly different from the one in \cite{antieau-williams-2014}.

\begin{proposition}
\label{proposition:3.2}
Up to degree $5$, the integral cohomology of $BPU(p)$ is given by 
\[
\begin{array}{rcll}
H^i(BPU(p); \mathbb{Z})&=& \{0\} & \mbox{for $i=1, 2, 5$, }
\\
H^i(BPU(p); \mathbb{Z})&=& \mathbb{Z}/p & \mbox{for $i=3$,}
\\
H^i(BPU(p); \mathbb{Z})&=& \mathbb{Z} & \mbox{for $i=4$.}
\end{array}
\]
Up to degree $5$, the mod $p$ cohomology of $BPU(p)$ is 
given by
\[
\begin{array}{rcll}
H^i(BPU(p))&=& \{0\} & \mbox{for $i=1, 5$, }
\\
H^i(BPU(p))&=& \mathbb{Z}/p & \mbox{for $i=0,2, 3, 4$.}
\end{array}
\]
\end{proposition}

\begin{proof}
Consider the Leray-Serre spectral sequence associated with 
\[
BU(p) \to BPU(p) \to K(\mathbb{Z}, 3)
\]
converging to $H^{*}(BPU(p); \mathbb{Z})$.
The integral cohomology of $BU(p)$ is a polynomial algebra 
generated by Chern classes, that is, 
$H^{*}(BU(p); \mathbb{Z})=\mathbb{Z}[c_1, \dots, c_p]$, where $\deg c_i=2i$.
The integral cohomology $H^i(K(\mathbb{Z},3); \mathbb{Z})$
of the Eilenberg-MacLane space $K(\mathbb{Z},3)$ is 
$\mathbb{Z}$ for $i=0,3$, $\{ 0\}$ for $i=1,2,4,5$. 
We fix a generator $u_3$ of $H^3(K(\mathbb{Z},3);\mathbb{Z})$.
Up to degree $5$, only non-trivial $E_2$-terms are
\[
E_{2}^{0,0}=E_2^{0,2}=\mathbb{Z}, \quad E_{2}^{0,4}=\mathbb{Z}\oplus \mathbb{Z}
\]
and
\[
E_2^{3,0}=E_2^{3,2}=\mathbb{Z}.
\]
Hence, up to degree $5$, the only non-trivial differential is $d_3:E_3^{0,t}\to E_3^{3,t-2}$.
The differential are given by
\[
d_3(c_1)=\alpha_1 u_3, \quad d_3(c_2)=\alpha_2 c_1u_3,
\]
where $\alpha_1, \alpha_2$ in $\mathbb{Z}$.
Since $BPU(p)$ is simply connected and since $\pi_2(BPU(p))=\mathbb{Z}/p$, 
by the Hurewicz theorem, we have $H_1(BPU(p);\mathbb{Z})=\{0\}$ and 
$H_2(BPU(p);\mathbb{Z})=\mathbb{Z}/p$.
By the universal coefficient theorem, we have
that $H^2(BPU(p);\mathbb{Z})=\{0\}$  and that 
 $H^3(BPU(p);\mathbb{Z})$ has $\mathbb{Z}/p$ as a direct summand.
 Therefore, 
$\alpha_1$ must be $\pm p$ and  $E_3^{3,0}=\mathbb{Z}/p$.
The cohomology suspension $\sigma:H^4(BU(p))\to H^3(U(p))$ maps $\rho(c_2)$ 
to a non-trivial primitive element in 
$H^{3}(U(p))$, 
but there exists no primitive element in $H^3(PU(p))$
by the computation due to Baum and Browder in \cite{baum-browder-1965}.
Hence,  in 
the Leray-Serre spectral sequence $E_r^{s,t}(BSU(p))$, 
the element $\rho(c_2) \in E_{2}^{0,4}(BSU(p))$ must support 
a non-trivial differential. Therefore,  $\alpha_2$ is not divisible by $p$ and
up to degree $5$, 
the non-trivial $E_3$-terms are
\[
E_{3}^{0,0}=E_{3}^{0,4}=\mathbb{Z}, 
\quad
E_{3}^{3,0}=\mathbb{Z}/p.
\]
As for $E_r^{s,t}(BPU(p))$, 
we have 
\[
E_{2}^{0,0}(BPU(p))=E_2^{0,2}(BPU(p))=\mathbb{Z}/p, \quad 
E_{2}^{0,4}(BPU(p))=\mathbb{Z}/p\oplus \mathbb{Z}/p, 
\]
\[
E_2^{3,0}(BPU(p))=E_2^{3,2}(BPU(p))=\mathbb{Z}/p, 
\]
and
\[
d_3(\rho(c_1))=0, \quad d_3(\rho(c_2))= \rho(\alpha_2 c_1u_3)\not=0.
\]
So, we have the desired result.
\end{proof}

With the following proposition, we choose generators 
\[
z_1\in H^1(B\langle \xi\rangle), 
\quad 
u_2\in H^2(BPU(p))
\]
 such that 
 \[
 d_2(z_1)=u_2, \quad d_2(z_1)=x_1y_1
 \]
  in the spectral sequences  associated with vertical fibre bundles in 
  (\ref{equation:3.1}).

\begin{proposition}\label{proposition:3.3}
We may choose $u_2\in H^2(BPU(p))$ such that
the induced homomorphism $\iota^*:H^{2}(BPU(p))\to H^2(BA_2)$ 
maps $u_2$ to $x_1y_1$.
\end{proposition}

\begin{proof}
From the  commutative diagram (\ref{equation:3.1}),
there exists the induced homomorphism 
between the Leray-Serre spectral sequences 
\[
\iota^*:E_r^{s,t}(BSU(p)) \to E_r^{s,t}(Bp_+^{1+2}).
\]
Since the group extension 
\[
\mathbb{Z}/p \to p_{+}^{1+2} \to A_2
\]
corresponds to $x_1y_1$ in $H^2(BA_2)$, 
the differential $d_2:E_2^{0,1}(Bp_+^{1+2}) \to E_2^{2,0}(Bp_+^{1+2})$
is
given by 
\[
d_2(z_1)=x_1y_1
\]
for some $z_1\in H^1(B\langle \xi \rangle)=\mathbb{Z}/p[z_2] \otimes \Lambda(z_1)$.
Hence, 
\[
d_2:E_2^{0,1}(BSU(p)) \to E_2^{2,0}(BSU(p))
\]
is nontrivial and we may define $u_2$ by $d_2(z_1)$.
Hence, we have the desired result.
\end{proof}

We end this section by computing $H^4(BG;\mathbb{Z})$
for $G=SU(p)\times SU(p)/\langle \Delta(\xi)\rangle$.
The following computation was done by Totaro 
 in the proof of Theorem 15.4 in \cite{totaro-2014}.


\begin{proposition} \label{proposition:3.4}
Consider a homomorphism
\[
\psi:H^4(BG; \mathbb{Z})\to H^4(BPU(p); \mathbb{Z}) \oplus H^4(BSU(p); \mathbb{Z})
\]
sending $x$ to $(\Delta^*(x), \Gamma_2^*(x))$. 
It is an isomorphism.
\end{proposition}

\begin{proof}
Let $p_1:PU(p)\times PU(p) \to PU(p)$ be the projection onto the first factor.
Then, $p_1\circ \pi$'s fibre is $SU(p)$. Consider the spectral sequence associated with 
\[
BSU(p) \stackrel{\Gamma_2}{\longrightarrow}
BG \stackrel{p_1\circ \pi}{\longrightarrow}
BPU(p).
\]
The $E_2$-term is $H^{s}(BPU(p);H^{t}(BSU(p); \mathbb{Z}))$.
By Proposition~\ref{proposition:3.2}, $E_2^{s,t}=\{0\}$ unless $s=0,3,4$ and $t=0,4$
up to degree $5$.
In particular, $E_2^{s,t}=\{0\}$ for $s+t=5$.
The non-zero $E_2$-terms of total degree $4$ are given by
\[E_2^{4,0}=\mathbb{Z},\quad
E_2^{0,4}=\mathbb{Z}.
\]
The non-zero $E_2$-term of total degree $3$ is given by
\[
E_2^{3, 0}=\mathbb{Z}/p.
\]
So, for dimensional reasons, 
we have $E_\infty^{s,t}=E_2^{s,t}$ for $s+t=4$.
Hence, we have $H^4(BG;\mathbb{Z})=\mathbb{Z}\oplus \mathbb{Z}$ and
a short exact sequence
\[
0 \to H^4(BPU(p);\mathbb{Z}) \stackrel{(p_1\circ \pi)^*}{\longrightarrow}
H^4(BG;\mathbb{Z}) \stackrel{{\Gamma_2}^*}{\longrightarrow}
H^4(BSU(p);\mathbb{Z}) \to 0.
\]
Since the composition $p_1\circ \pi \circ \Delta$ is the identity map, 
this short exact sequence splits and the homomorphism $\phi$
is an isomorphism.
\end{proof}


\section{The mod $p$ cycle map for $G$}\label{section:4}

Let $G=SU(p)\times SU(p)/\langle \Delta(\xi) \rangle$ as in Section~\ref{section:2}.
In this section, we  define a virtual complex representation $\lambda''$ of $G$.
Using the Chern class $c_2(\lambda'')$, we prove Theorem~\ref{theorem:1.2}
for $K=G$. To be precise, we show that $c_2(\lambda'')$ is non-zero in $CH^2BG/p$ and
the mod $p$ reduction maps
$c_2(\lambda'')$ to $0$ in $H^4(BG)$. 
Theorem~\ref{theorem:1.2} for $K=G$ was 
obtained by Totaro in \cite{totaro-2014} 
and by the author in \cite{kameko-2015} independently.
From now on, we denote the Bockstein operation 
of degree $1$ by $Q_0$ and the Milnor operation of degree $2p-1$ by $Q_1$. 
These are cohomology operations on the mod $p$ cohomology.

Let 
$\lambda_1:SU(p) \to U(p)$ be the tautological representation, so that 
$\lambda_1(g)(v)=gv$ for $v \in \mathbb{C}^p$.
Let 
\[
\lambda_1^* \otimes \lambda_1:SU(p) \times SU(p) \to U(p^2)
\]
be  the complex representation defined by
\[
(\lambda_1^*\otimes \lambda_1)(g_1, g_2)(v_1^* \otimes v_2)
=( v_1^*g_1^{-1}) \otimes (g_2v_2),
\]
where $\mathbb{C}^{p^2}=(\mathbb{C}^{p})^* \otimes \mathbb{C}^p$, 
$(\mathbb{C}^p)^*=\mathrm{Hom}(\mathbb{C}^p, \mathbb{C})$.
The complex representation $\lambda_1^*\otimes \lambda_1$ induces 
a complex representation
$\lambda:G \to U(p^2)$.
We define 
a complex representation $\lambda'$ by $\lambda \circ \Delta \circ p_1 \circ \pi$.
Using the complex representations  $\lambda$ and $\lambda'$, we define 
a virtual complex representation $\lambda''$ by
 $\lambda''=\lambda-\lambda'$.
An element in the complex representation ring of $G$ corresponds to an element in
the topological $K$-theory $K^0(BG)=[BG, \mathbb{Z}\times BU]$. 
By abuse of notation, 
we denote by 
$\lambda'':BG\to \mathbb{Z} \times BU$ a map in the  homotopy class corresponding to $\lambda''$.
It is clear that 
\[
\Delta^* (\lambda'')=0, \quad {\Gamma_2}^*(\lambda'')=p \lambda_1
\]
in the complex representation ring of $G$.

We denote by $x_4$ the cohomology class in $H^4(BG; \mathbb{Z})$ such that 
\begin{enumerate}
\item ${\Gamma_2}^*(x_4)=c_2(\lambda_1)$, 
\item $\Delta^{*}(x_4)=0$.
\end{enumerate}
Then, $c_2(\lambda'')=p x_4$. Hence, $\rho(c_2(\lambda''))=0$
in $H^4(BG)$.
It is  clear from the definition that 
$c_2(\lambda'') \not=0$ in $H^4(BG; \mathbb{Z})$.
Thus, if we show that the Chern class 
$c_2(\lambda'')$ is not divisible by $p$ in $CH^2BG$, then 
$c_2(\lambda'')$
represents a non-zero element in $CH^2BG/p$ and the mod $p$ cycle map
is not injective for $BG$.
We prove it by contradiction.
Suppose that the Chern class $c_2(\lambda'')$ is 
divisible by $p$, that is, we suppose that there exists
a virtual complex representation $\mu:BG\to \mathbb{Z}\times BU$ of $G$ such that
$x_4 \in \mathrm{Im}\, \mu^* \subset H^4(BG; \mathbb{Z})$. 
Then, $Q_1\rho(x_4)$ must be zero since $H^{odd}(\mathbb{Z} \times BU)=\{0\}$.
We prove the non-existence of the above virtual complex representation
by showing that $Q_1\rho(x_4)\not=0$.
To show that $Q_1\rho(x_4)\not=0$, 
we show that $Q_1(f\circ g)^{*}(\rho(x_4))\not=0$ in $H^{*}(BA_3)$, where
$f, g$ and $A_3$ are defined in Section~\ref{section:2}.
The following Proposition~\ref{proposition:4.1} completes 
the proof of Theorem~\ref{theorem:1.2} for $K=G$.

We proved $(f\circ g)^*(\rho(x_4))=Q_0(x_1y_1z_1)$ in  \cite{kameko-2015}.
Because we use a similar but slightly different  
argument in the proof of 
Theorem~\ref{theorem:1.2} for $K=H$, 
we prove the following  weaker form in this paper.

\begin{proposition}\label{proposition:4.1}
We have
$Q_1(f\circ g)^*(\rho(x_4))\not=0$ in $H^{2p+3}(BA_3)$.
\end{proposition}

To prove Proposition~\ref{proposition:4.1}, 
we compute  the Leray-Serre spectral sequences  and the homomorphism 
$(f\circ g )^*$ induced by the following commutative diagram.
\[
\begin{diagram}
\node{BA_3} \arrow{s,l}{\varphi} \arrow{e,t}{f\circ g}
\node{BG} \arrow{s,r}{\pi}
\node{BSU(p)} \arrow{s,r}{\tilde{\pi}} \arrow{w,t}{\Gamma_2}
\\
\node{BA_2}\arrow{e,t}{f\circ g}
\node{BPU(p)\times BPU(p)}
\node{BPU(p)} \arrow{w,t}{\Gamma_2}
\end{diagram}
\]
We denote by $x_1, y_1$ the generators of the mod $p$ cohomology of $BA_3$
corresponding to the generators  
$\Delta(\alpha)$, 
$ \Delta(\beta)$
of $A_3$, so that we have $\varphi^*(x_1)=x_1$, $\varphi^*(y_1)=y_1$.
Let $z_1$ be the element in $H^1(B\langle \Gamma_2(\xi) \rangle)$ such that 
${\Gamma_2}^*(z_1)=-z_1 \in E_2^{0,1}(BSU(p))$. The element $z_1$ in 
$E_2^{0,1}(BSU(p))$ and $u_2\in E_2^{2,0}(BSU(p))$ are
 defined in Section~\ref{section:3}, so that $d_2(z_1)=u_2$ in $E_2^{2,0}(BSU(p))$.
 We define the generator $u_3$ of $H^3(BPU(p))$ by $u_3=Q_0 u_2$.
Let us consider the $E_2$-term of the spectral sequence $E_r^{s,t}(BG)$. 
The $E_2$-term is as follows:
\[
E_2^{*,*}=H^{*}(BPU(p))\otimes H^{*}(BPU(p)) \otimes \mathbb{Z}/p[z_2]\otimes \Lambda(z_1).
\]
Since $f\circ g=\Delta \circ \iota$, we have $(f\circ g )^*(1\otimes u)=(f\circ  g)^*(u\otimes 1)=\iota^*(u)$.
Moreover, we have ${\Gamma_2}^*(1\otimes u)=u$ and ${\Gamma_2}^{*}(u \otimes 1)=0$ for $\deg u>0$.

Let $a_i=u_i\otimes 1-1\otimes u_i$, $b_i=u_i\otimes 1$.
Then, up to degree $6$, the $E_2$-term is a free $\mathbb{Z}/p[a_2, z_2]\otimes \Lambda(z_1)$-module with the basis 
$\{ 1, b_2, a_3, b_3, b_2^2, a_3b_3, b_2^3 \}$.
Since $(f\circ g)^*d_2(z_1)=0$ and ${\Gamma_2}^*(d_2(z_1))=-u_2$, 
the first non-trivial differential  is given by
\[
d_2(z_1)=a_2.
\]
So, up to degree $5$, the $E_3$-term is a free $\mathbb{Z}/p[z_2]$-module with the basis 
$\{ 1, b_2, a_3, b_3, b_2^2 \}$.
In particular, we have $a_3b_2=0$ in $E_3^{5,0}$.
Since $(f\circ g)^*(d_3(z_2))=0$ and ${\Gamma_2}^*(d_3(z_2))=-u_3$, 
the second non-trivial differential is given by 
\[
d_3(z_2)=a_3.
\]
Up to degree $4$, the $E_4$-term is a free $\mathbb{Z}/p$-module with the basis $\{ 1, b_2, b_3, b_2^2, b_2z_2 \}$
and the spectral sequence  collapses at the $E_4$-level. 
Thus, the $E_\infty$-terms of total degree $4$ are as follows:
\[
\begin{array}{lllll}
E_\infty^{0,4}=\{ 0\}, 
& E_\infty^{1,3}=\{ 0\},
& E^{2,2}_\infty =\mathbb{Z}/p \{ b_2 z_2 \}, 
&E_\infty^{3,1}=\{ 0 \},
&
E_{\infty}^{4,0}=\mathbb{Z}/p \{ b_2^2\}.
\end{array}
\]
The element $b_2$ is a permanent cocycle. 
By abuse of notation, we denote by ${b}_2$ the cohomology class in 
$F^2H^2(BG)$ representing $b_2$.
Since 
$
H^2(BSU(p))=\{ 0\},
$
we have
\[
{\Gamma_2}^*(\pi^*({b}_2))=0.
\]
Moreover, $\pi^{*}(H^4(BPU(p)\times BPU(p)))=\mathbb{Z}/p\{ b_2^2\}.$
Hence, we have
\[
{\Gamma_2}^{*}(\pi^*(H^4(BPU(p)\times BPU(p))))=\{ 0\}.
\]

On the other hand,  ${\Gamma_2}^*\rho(x_4)=\rho(c_2(\lambda_1))\not=0$ in $H^4(BSU(p))$.
Therefore,  $\rho(x_4)$ is not in the image of 
\[
\pi^*:H^4(BPU(p)\times BPU(p)) \to H^4(BG).
\]
Hence, we have the following result:

\begin{proposition}
\label{proposition:4.2}
The cohomology class $\rho(x_4)$ represents $\alpha b_2 z_2$ in $E_{\infty}^{2,2}$
 for some $\alpha\not=0$ in $\mathbb{Z}/p$.
\end{proposition}

Now, we  complete the proof of Proposition~\ref{proposition:4.1} 
using Proposition~\ref{proposition:4.2}.

\begin{proof}[Proof of Proposition~\ref{proposition:4.1}]
Since $(f\circ g)^*(b_2)=x_1y_1$, we have 
\[
(f\circ g)^*(b_2z_2)=x_1y_1z_2
\]
in the spectral sequence, where $z_2=Q_0 z_1$ in $H^2(B\langle \Gamma_2(\xi)\rangle)$.
Let $x_2=Q_0 x_1$, $y_2=Q_0 y_1$. Then, $H^{*}(BA_3)=\mathbb{Z}/p[x_2, y_2, z_2]\otimes \Lambda(x_1, y_1, z_1)$. $\varphi^*(H^*(BA_2))$ is the subalgebra generated by $x_1, y_1, x_2, y_2$.
Therefore, we have
\[
(f\circ g)^*(\rho(x_4))=\alpha x_1y_1 z_2+ u'z_1+u''
\]
 for some $u', u'' \in \varphi^*(H^*(BA_2))$.
Let $M$ be the $\varphi^*(H^*(BA_2))$-module generated by
\[
1, \; z_2^{i}, \; z_1, \; z_1z_2,  \; z_1z_2^i (i\geq 2), 
\]
so that 
\[
H^{*}(BA_3)/M=\varphi^*(H^{*}(BA_2)) \{ z_2 \}.
\]
Since $Q_1 z_1=z_2^p$, $Q_1z_2=0$, and $Q_1$ is a derivation, $M$ is closed under the action of 
Milnor operation $Q_1$.
We have
\[
(f\circ g)^*(\rho(x_4))\equiv \alpha x_2^p y_1 z_2- \alpha x_1y_2^p z_2 \not\equiv 0 
\mod M.
\]
It completes the proof of Proposition~\ref{proposition:4.1}.
\end{proof}

\section{The mod $p$ cohomology of $B\tilde{\pi}(H_2)$} \label{section:5} 

In this section, we collect some facts on the mod $p$ cohomology of $B\tilde{\pi}(H_2)$
as Propositions~\ref{proposition:5.1} and \ref{proposition:5.2}.
We use these facts in the proof of Proposition~\ref{proposition:6.1} in the next section.

We begin by defining generators of $H^1(B\tilde{\pi}(H_2))$. 
Since the commutator subgroup $[\tilde{\pi}(H_2), \tilde{\pi}(H_2)]$ is generated by 
$\tilde{\pi}(\mathrm{diag}(\xi^{a_1}, \dots, \xi^{a_p}))$, where $a_1, \dots, a_p$ range over $\{
0,\dots, p-1\}$ with $a_1+\cdots+a_p\equiv 0 \mod  p$, 
\[
\tilde{\pi}(H_2)/[\tilde{\pi}(H_2), \tilde{\pi}(H_2)]=\mathbb{Z}/p\oplus \mathbb{Z}/p
.
\]
This elementary abelian $p$-group is generated by 
$\tilde{\pi}(\sigma_1), \tilde{\pi}(\beta)$.
We denote by $v_1, w_1$ the generators of 
$H^1(B\langle \tilde{\pi}(\sigma_1) \rangle)$, 
$H^1(B\langle \tilde{\pi}(\beta) \rangle)$ corresponding to $\tilde{\pi}(\sigma_1)$, 
$\tilde{\pi}(\beta)$.
By abuse of notation, we denote the corresponding generators in 
$H^1(B\tilde{\pi}(H_2))$ by the same symbol, so that
for the inclusions 
\[
\iota_{\beta}:\langle \tilde{\pi}(\beta) \rangle \to 
\tilde{\pi}(H_2),\quad \iota_{\sigma}:\langle \tilde{\pi}(\sigma_1) \rangle \to 
\tilde{\pi}(H_2),
\]
we have 
$\iota_\beta^*(w_1)=w_1$, $\iota_\beta^*(v_1)=0$, 
$\iota_\sigma^*(w_1)=0$, $\iota_\sigma^*(v_1)=v_1$.
Indeed, we have 
$H^{*}(B\langle \tilde{\pi}(\sigma_1)\rangle)=\mathbb{Z}/p[v_2]\otimes \Lambda(v_1)$, 
$H^{*}(B\langle \tilde{\pi}(\beta)\rangle )=\mathbb{Z}/p[w_2]\otimes \Lambda(w_1)$, 
where
$v_2=Q_0 v_1$, $w_2=Q_0w_1$.
We denote the inclusion of $\tilde{\pi}(H_2)$ to $PU(p)$ by 
\[
\iota:\tilde{\pi}(H_2)\to PU(p)
\]
and we recall that we defined the generator $u_2$ 
of $H^{2}(BPU(p))$ in Proposition~\ref{proposition:3.3}.

\begin{proposition}\label{proposition:5.1}
In $H^{*}(B\tilde{\pi}(H_2))$, 
we have
$\iota^*(u_2) v_1\not=0$, $\iota^*(u_2^2)\not=0$.
\end{proposition}
\begin{proof}
We consider the Leray-Serre spectral sequences associated with 
the vertical fibrations in the following commutative diagram.
\[
\begin{diagram}
\node{B\langle \sigma_1\rangle} \arrow{e,t}{\iota_\sigma} \arrow{s,l}{\tilde{\pi}}
\node{BH_2} \arrow{s,r}{\tilde{\pi}}
\arrow{e,t}{\iota}
\node{BSU(p)}
\arrow{s,r}{\pi}
\\
\node{B\langle \tilde{\pi}(\sigma_1)\rangle}  \arrow{e,t}{\iota_\sigma}
\node{B\tilde{\pi}(H_2)}
\arrow{e,t}{\iota}
\node{BPU(p).}
\end{diagram}
\]
Let $z_1 \in E_2^{0,1}(BSU(p))$, $u_2\in E_2^{2,0}(BSU(p))$ 
be elements defined in Section~\ref{section:3}. By abuse of notation, 
we denote elements 
$\iota^*(z_1) \in E_2^{0,1}(BH_2)$, 
$\iota_\sigma^*(\iota^*(z_1)) \in E_2^{0,1}(B\langle \sigma_1\rangle)$
by $z_1$.
Since
$\langle \sigma_1\rangle =\mathbb{Z}/p^2$, 
\[
d_2(z_1)=\alpha v_2
\]
 for some $\alpha \not=0$ in $\mathbb{Z}/p$
  in the Leray-Serre spectral sequence 
$E_2^{2,0}(B\langle \sigma_1\rangle)$.
Since $u_2= d_2(z_1)$ in the Leray-Serre spectral sequence 
$E_2^{2,0}(BSU(p))$, 
we have
\[
\iota_\sigma^*(\iota^*(u_2))= d_2(z_1)=\alpha v_2
\]
 in 
$
H^{*}(B\langle \tilde{\pi}(\sigma_1)\rangle)=\mathbb{Z}/p[v_2]\otimes \Lambda(v_1).
$
Hence, we have 
$\iota_\sigma^*(\iota^*(u_2) v_1)=\alpha v_1v_2\not=0$, 
$\iota_\sigma^*(\iota^*(u_2^2))=\alpha^2 v_2^2\not=0$.
Therefore, we obtain the desired result
$\iota^{*}(u_2) v_1\not=0$ and 
$\iota^{*}(u_2^2)\not=0$ in $H^{*}(B\tilde{\pi}(H_2))$.
\end{proof}

\begin{proposition}\label{proposition:5.2}
In $H^{*}(B\tilde{\pi}(H_2))$, 
we have
$\iota^*(u_2) w_1=0$.
\end{proposition}

To prove Proposition~\ref{proposition:5.2}, we defined the subgroup 
$H_2'=T^p \rtimes \mathbb{Z}/p$ of the unitary group $U(p)$
generated by  diagonal matrices and $\beta$. 
The quotient group $\tilde{\pi}'(H_2')$ contains $\tilde{\pi}(H_2)$ as a subgroup
and they are subgroups of the projective unitary group $PU(p)$.
We denote by 
\[
\iota'':\tilde{\pi}(H_2)\to \tilde{\pi}'(H_2'), \quad \iota':\tilde{\pi}'(H_2')\to PU(p)
\]
 the inclusions, so that $\iota=\iota' \circ \iota''$.
We use the following lemma in the proof of Proposition~\ref{proposition:5.2}.

\begin{lemma}\label{lemma:5.3}
In $H^{*}(B\tilde{\pi}'(H_2'))$, 
there exists an element  $t_2\in H^2(B\tilde{\pi}'(H_2'))$ such that
$H^{1}(B\tilde{\pi}'(H_2'))=\mathbb{Z}/p\{ w_1\}$, 
$H^{2}(B\tilde{\pi}'(H_2'))=\mathbb{Z}/p\{ t_2, w_2\}$
where $w_2=Q_0 w_1$, $(\iota'' \circ \iota_{\sigma})^*(t_2)=v_2$ and $(\iota''\circ \iota_\beta)^*(t_2)=0$.
Moreover, we have $t_2 w_1=0$ in $H^*(B\tilde{\pi}'(H_2'))$.
\end{lemma}

Now, we prove Proposition~\ref{proposition:5.2} assuming Lemma~\ref{lemma:5.3}.

\begin{proof}[Proof of Proposition~\ref{proposition:5.2}]
We consider the Leray-Serre spectral sequences associated with 
the vertical fibrations in the following commutative diagram.
\[
\begin{diagram}
\node{B\langle \beta, \xi \rangle} \arrow{e,t}{\iota_\beta} \arrow{s,l}{\tilde{\pi}}
\node{BH_2} \arrow{s,r}{\tilde{\pi}}
\arrow{e,t}{\iota}
\node{BSU(p)}
\arrow{s,r}{\tilde{\pi}}
\\
\node{B\langle \tilde{\pi}(\beta)\rangle}  \arrow{e,t}{\iota_\beta}
\node{B\tilde{\pi}(H_2)}
\arrow{e,t}{\iota}
\node{BPU(p).}
\end{diagram}
\]
Suppose that 
${\iota'}^{*}(u_2)=\alpha_1 t_2+\alpha_2 w_2$ where $\alpha_1, \alpha_2\in \mathbb{Z}/p$.
Then, by Lemma~\ref{lemma:5.3}, 
we have 
\[
{\iota'}^*(u_2) w_1=\alpha_{1} t_2w_1+\alpha_{2} w_1w_2=\alpha_2 w_1w_2.
\]
Hence, we have $(\iota \circ \iota_\beta)^*(u_2)w_1=\alpha_{2} w_1w_2$.
On the other hand, since the group extension 
\[
\langle \xi \rangle \to \langle \beta, \xi \rangle \to \langle \tilde{\pi}(\beta) \rangle
\]
is trivial, $d_2:H^1(B\langle \xi \rangle) \to H^2(B\langle \tilde{\pi}(\beta)\rangle)$
in $E_2^{2,t}(B\langle \beta, \xi\rangle)$ is zero and 
\[
(\iota \circ \iota_\beta)^*(u_2)=d_2((\iota\circ \iota_\beta)^*(z_1))=0
\]
 in $H^{*}(B\langle \tilde{\pi}(\beta) \rangle)=E_2^{2,0}(B\langle \beta, \xi\rangle)$.
Therefore, we have $\alpha_2=0$ and  $w_1 {\iota'}^*(u_2)=0$ in 
$H^{*}(B\tilde{\pi}'(H'_2))$. 
Therefore, 
we have 
\[
\iota^*(u_2)w_1={\iota''}^*( {\iota'}^*(u_2) w_1)=0
\]
in $H^{*}(B\tilde{\pi}(H_2)).$
\end{proof}


We end this section by proving Lemma~\ref{lemma:5.3}.

\begin{proof}[Proof of Lemma~\ref{lemma:5.3}]
We need to study the mod $p$ cohomology only up to degree $3$.
We  define $t_2$ by ${\iota'}^*(u_2)$ where $u_2$ is the generator of $H^2(BPU(p))$. 

We consider the Leray-Serre spectral sequence associated with 
the following commutative diagram:
\[
\begin{diagram}
\node{BT^{p}} \arrow{s}\arrow{e,t}{{\tilde{\pi}}'} \node{BT^{p-1}} \arrow{s} \\
\node{BH_2'} \arrow{s} \arrow{e,t}{{\tilde{\pi}}'} \node{B\tilde{\pi}'(H_2')} \arrow{s} \\
\node{B\langle \beta \rangle} \arrow{e,t}{\tilde{\pi}}
 \node{B\langle \tilde{\pi}(\beta) \rangle.}
\end{diagram}
\]
We choose a generator $t_2^{(i)} \in H^2(BT^p)$ 
corresponding to the $i$-th diagonal entry of $T^p$, so that
$H^{2}(BT^p)=\mathbb{Z}/p\{ t_2^{(1)}, \dots, t_2^{(p)}\}$.
The matrix $\beta$ acts on $T^p$ as the cyclic permutation of diagonal entries, so that
it acts on $H^{2}(BT^{p})$ as the cyclic permutation on
$t_2^{(1)}, \dots, t_2^{(p)}$. The induced homomorphism
 ${ \tilde{\pi}}'^*:H^{2}(BT^{p-1}) \to H^{2}(BT^p)$ is injective and
 we may take  a basis $\{ u_2^{(1)}, \dots, u_2^{(p-1)} \}$ for $H^2(BT^{p-1})$ 
 such that 
 $\tilde{\pi}'^*(u_2^{(i)})=t_2^{(i)}-t_2^{(i+1)}$ for $i=1, \dots, p-1$. 
 Hence, $\langle \beta \rangle$ acts on $H^2(BT^{p-1})$ 
 by 
 \[
 g u_2^{(i)}=u_2^{(i+1)}
 \]
  for $i=1,\dots, p-2$  and
\[
g u_2^{(p-1)}=-(u_2^{(1)}+\cdots+u_2^{(p-1)})
\]
 for some generator $g$ of $\langle \beta \rangle$.
We consider the Leray-Serre spectral sequence converging to the mod $p$
cohomology of $B\tilde{\pi}'(H_2')$.
The $E_1$-term is additively given as follows:
\[
E_1=\mathbb{Z}/p[u_2^{(1)}, \dots, u_2^{(p-1)}]
\{ \; w_2^{i}, w_1w_2^i\;|\; i\geq 0 \; \}.
\]
The first non-trivial differential is given by 
\[
d_1(u w_2^i)=((1-g) u) w_1w_2^{i}, \quad
d_1(u  w_1w_2^i)=((1-g)^{p-1} u )w_2^{i+1},
\]
where $u\in \mathbb{Z}/p[u_2^{(1)}, \dots, u_2^{(p-1)}]=E_1^{0,*}$.
The kernel of 
\[
(1-g):\mathbb{Z}/p\{ u_2^{(1)}, \dots, u_2^{(p-1)} \} 
\to \mathbb{Z}/p\{ u_2^{(1)}, \dots, u_2^{(p-1)} \} 
\]
 is
spanned by a single element 
\[
u_2^{(1)}+2u_2^{(2)}+\cdots +(p-1) u_2^{(p-1)}
\]
and the image of $(1-g)$ is spanned by $p-2$ elements
\[
u_{2}^{(1)}-u_{2}^{(2)}, \; \dots, \;  u_2^{(p-2)}-u_2^{(p-1)}.
\]
We denote the generator of the kernel of $(1-g)$ by $\tilde{u}$, that is, 
\[
\tilde{u}=u_2^{(1)}+2u_2^{(2)}+\cdots +(p-1) u_2^{(p-1)}.
\]
It is easy to see that
\[
\tilde{u}
\equiv
 (1+\cdots +(p-1))u_{2}^{(p-1)} 
 \equiv \dfrac{p(p-1)}{2} u_2^{(p-1)} 
\equiv 0
\]
modulo the image of $(1-g)$.
By direct calculation, we have $(1-g)^{p-1}(u_2^{(1)})=0$ and $\mathrm{Ker}\, (1-g)^{p-1}=\mathbb{Z}/p\{ u_2^{(1)}, \dots, u_2^{(p-1)}\}$.
Hence, we have 
\begin{align*}
E_2^{0,2}&=\mathrm{Ker}\, (1-g)=\mathbb{Z}/p\{ \tilde{u}\} , 
\\
 E_2^{1,2}&=(\mathrm{Ker}\, (1-g)^{p-1}/\mathrm{Im}\, (1-g))\{ w_1\} =\mathbb{Z}/p\{ u_2^{(1)}w_1 \} , 
\end{align*}
respectively.
Moreover, we have 
$E_r^{*, odd}=\{ 0\}$, $E_r^{*,0}=\mathbb{Z}/p[w_2]\otimes \Lambda(w_1)$ for $*\geq 0$, 
$r\geq 1$.
Since the elements in $E_r^{*,0}$ are permanent cocycles, 
the spectral sequence collapses at the $E_2$-level up to degree $3$.
Choose a cohomology class $t_2'$ in $H^2(B\tilde{\pi}'(H_2'))$
representing the generator $\tilde{u}$ of $E_\infty^{0,2}=\mathbb{Z}/p$. 
Then, $H^2(B\tilde{\pi}'(H_2'))$ is generated by $t_2'$ and $w_2$.
Suppose that \[
{\iota'}^*(u_2)=\alpha_1 w_2+\alpha_2 t_2', 
\]where $\alpha_1, \alpha_2\in \mathbb{Z}/p$.
Since $(\iota'\circ \iota'' \circ \iota_\sigma)^*(u_2)=v_2$
and $(\iota''\circ \iota_\sigma)^*(w_2)=0$, 
\[
(\iota''\circ \iota_\sigma)^*(\alpha_2 t_2')=v_2
\] and so $\alpha_2\not=0$.
Hence, $t_2$ and $w_2$ generate $H^2(B\tilde{\pi}'(H_2'))$.

Next, we prove that $t_2w_1=0$.
The $E_\infty$-terms of total degree $3$ are given by
\[
E_\infty^{0,3}=\{ 0 \}, \quad E_\infty^{1,2}=\mathbb{Z}/p\{ u_2^{(1)} w_1 \}, \quad 
E_\infty^{2,1}=\{0\} \quad \mbox{and} \quad 
E_\infty^{3,0}=\mathbb{Z}/p\{ w_1w_2\}.
\]
Therefore, we have
\[
F^2H^3(B\tilde{\pi}'(H_2'))=F^3H^3(B\tilde{\pi}'(H_2'))
=\mathbb{Z}/p\{ w_1w_2\}. 
\]
Since $\alpha_2 t_2' w_1$ represents $\alpha_2\tilde{u}  w_1$ 
and since $\tilde{u}\in \mathrm{Ker}\, (1-g)$ is congruent to zero modulo the image of $(1-g)$, 
we have
$\tilde{u}w_1=0$ in $E_\infty^{1,2}$.
So, we have
\[
t_2 w_1 \in 
F^3H^3(B\tilde{\pi}'(H_2'))
=\mathbb{Z}/p\{ w_1w_2\}. 
\]
Therefore, 
$ t_2 w_1=\alpha_3 w_1w_2$ for some $\alpha_3\in \mathbb{Z}/p$. 
We proved that 
$(\iota''\circ \iota_\beta)^*(t_2)=(\iota'\circ\iota''\circ  \iota_\beta)^*(u_2)=0$
in the proof of Proposition~\ref{proposition:5.2}.
Thus, we have $(\iota''\circ \iota_\beta)^{*}( t_2w_1 )=0$.
On the other hand, we have 
$(\iota''\circ \iota_\beta)^*(w_1w_2)=w_1w_2\not=0$
in $H^{*}(B\langle \tilde{\pi}(\beta)\rangle )$.
Hence, we obtain $\alpha_3=0$.
\end{proof}

\section{The mod $p$ cycle map for $H$} \label{section:6} 

In this section, we prove  Theorem~\ref{theorem:1.2}.
Let $G$ be $SU(p)\times SU(p)/\langle \Delta(\xi) \rangle$
and let $H=p_{+}^{1+2} \times H_2/\Delta(\xi)$
 as in Section~\ref{section:3}.
 Let $K$ be a subgroup of $G$ containing $H$, that is, $H\subset K \subset G$.
We proved in Section~\ref{section:4} that the mod $p$ cycle map $CH^2BG/p \to H^4(BG)$ is
not injective. To be more precise,  we defined the virtual complex representation $\lambda'':BG\to \mathbb{Z} \times BU$ such that 
the Chern class 
$c_{2}(\lambda'') \in CH^2BG$ is non-trivial in $CH^2BG/p$, 
that is, $c_2(\lambda'')$ is not divisible by $p$, and 
the mod $p$ cycle map maps $c_2(\lambda'')$ to $\rho(c_2(\lambda''))=0$. 
We denote the inclusions by $f':K\to G$, $f'':H\to K$, $f:H\to G$, so that
$f=f'\circ f'':H\to G$.
It is clear that
$\rho(c_2(\lambda''\circ f'))$ is zero in $H^4(BK)$.
So, in order to prove Theorem~\ref{theorem:1.2}, 
we need to show that
 { $c_2(\lambda''\circ f')$  remains non-zero in $CH^2BK\subset H^4(BK;\mathbb{Z})$} 
and that $c_2(\lambda''\circ f')$ remains not divisible by $p$  in $CH^2BK$.
These follow immediately from the following:
\begin{itemize}
\item[(1)] { $c_2(\lambda''\circ f)={f''}^*(c_2(\lambda''\circ f'))$ 
 is not zero in $CH^2BH\subset H^4(BH;\mathbb{Z})$.} 
\item[(2)] $c_2(\lambda''\circ f)={f''}^*(c_2(\lambda''\circ f'))$
 is not divisible by $p$  in $CH^2BH$.
\end{itemize}
We prove  (1) and (2)  in the rest of this section.

To prove (1) and (2), we consider the spectral sequences associated with the vertical fibrations below
and the induced homomorphism between them.
\[
\begin{diagram}
\node{BH} \arrow{s,l}{\pi} \arrow{e,t}{f}
\node{BG} \arrow{s,r}{\pi}
\\
\node{BA_2\times B\tilde{\pi}(H_2)}\arrow{e,t}{f}
\node{BPU(p)\times BPU(p)}
\end{diagram}
\]

Let $g:BA_2\to BA_2\times B\tilde{\pi}(H_2)$ be the map defined in Section~\ref{section:2} 
by $g(\tilde{\pi}(\alpha))=(\tilde{\pi}(\alpha),\tilde{\pi}(\alpha))$, $g(\tilde{\pi}(\beta))=(\tilde{\pi}(\beta),\tilde{\pi}(\beta))$.
Let $v_1, w_1$ be the generators of  $H^1(B\tilde{\pi}(H_2))$ 
defined in the previous section.
Let  $x_1, y_1$ be those of $H^1(BA_2)$ as defined in Section~\ref{section:3}.
We denote by $x_1, y_1, v_1, w_1$ the corresponding generators of 
$H^1(BA_2\times B\tilde{\pi}(H_2))$,  so that 
$g^*(x_1)=x_1$, $g^*(v_1)=0$, $g^*(y_1)=g^{*}(w_1)=y_1$.
We denote by $z_1$ a generator of $H^1(B\langle \Gamma_2(\xi) \rangle)=E_2^{0,1}$
as in Section~\ref{section:4}. 
Let 
$x_2=Q_0 x_1$, $y_2=Q_0 y_1$, $z_2=Q_0 z_1$
as usual, so that
$H^{*}(BA_2)=\mathbb{Z}/p[x_2,y_2]\otimes \Lambda(x_1, y_1)$. 
Also, let $u_2$ be the generator
of $H^2(BPU(p))$ defined in Section~\ref{section:3}, and let $u_3 = Q_0 u_2$
as in Section~\ref{section:4}.
Let $\iota$ be the map induced by the inclusion of 
$\tilde{\pi}(H_2)$ to $PU(p)$. We need to compute the spectral sequence up to degree $4$.
Differentials $d_2$, $d_3$ in the spectral sequence $E_r^{s,t}(BH)$ are given by:
\begin{align*}
d_2(z_1)& =x_1y_1- \iota^*(u_2), \\
d_3(z_2)&=x_2y_1-x_1y_2-\iota^*(u_3), 
\end{align*}
since 
\begin{align*}
f^*(u_2\otimes 1 - 1 \otimes u_2)&=x_1y_1-\iota^*(u_2), \\
f^*(u_3\otimes 1 - 1 \otimes u_3)&=x_2y_1-x_1y_2-\iota^*(u_3)
\end{align*}
and since the differentials $d_2$, $d_3$ in the spectral sequence 
$E_r^{s,t}(BG)$ are given by $d_2(z_1)=u_2\otimes 1 - 1 \otimes u_2$, $d_3(z_2)=u_3\otimes 1 - 1 \otimes u_3$ as we saw in Section~\ref{section:4}.

\begin{proposition}\label{proposition:6.1}
The $E_\infty$-terms $E_\infty^{s,t}$ {\rm (}$s=0,1, 2$, $s+t=3, 4${\rm )}
for the spectral sequence $E_r^{s,t}(BH)$ are given as follows:
$E_\infty^{0,3}=E_\infty^{1,2}=\{ 0\}$, 
\[
E_\infty^{2,1}=\mathbb{Z}/p\{ w_1x_1z_1, w_1y_1z_1\},
\]
$E_\infty^{0,4}=E_\infty^{1,3}=\{ 0\}$, 
and 
\[
E_{\infty}^{2,2}=\mathbb{Z}/p\{ x_1y_1z_2, w_1x_1z_2, w_1y_1z_2\}.\]
\end{proposition}

\begin{proof}
For the sake of notational simplicity, 
let
\[
R=\mathbb{Z}/p[x_2, y_2] \otimes H^{*}(B\tilde{\pi}(H_2)), 
\]
so that 
\[ 
H^{*}(BA_2)\otimes H^{*}(B\tilde{\pi}(H_2))=R\{ 1, x_1, y_1, x_1y_1\}.
\]
The set $\{ v_1, w_1\}$ is a basis for $H^1(B\tilde{\pi}(H_2))$. 
We consider a basis for $H^{2}(B\tilde{\pi}(H_2))$. 
By Proposition~\ref{proposition:5.1}, we have $\iota^*(u_2)^2\not=0$.
We choose a basis $\{ m^{(i)},  \iota^*(u_2) \}$ 
 for $H^{2}(B\tilde{\pi}(H_2))$, where $1\leq i <\dim H^{2}(B\tilde{\pi}(H_2))$.
Here, we do not exclude the possibility that  $\{ m^{(i)}\}$ 
 could be the empty set.
 Then, the set
 $\{ m^{(i)},   \iota^*(u_2), x_2, y_2\}$ is a basis for the subspace of $R$ spanned by 
 elements of degree $2$ and the set
$\{  m^{(i)},  x_2 ,y_2 \}$ is a basis for the subspace of $R/(\iota^{*}(u_2))$
spanned by elements of degree $2$.
The set
\[
\{ v_1 ,w_1, x_1, y_1\}
\] is a basis for $E_2^{1,0}=H^1(BA_2\times B\tilde{\pi}(H_2))$ and the set
\[
\{ m^{(i)}, \iota^*(u_2), x_2, y_2, \; v_1x_1, v_1y_1, w_1x_1, w_1y_1, x_1y_1\}
\] is a basis for $E_2^{2,0}=H^2(BA_2\times B\tilde{\pi}(H_2))$.


First, we compute $E_3$-terms $E_{3}^{0,3}$, $E_3^{2,1}$ and $E_3^{1,3}$.
Let us consider  $R$-module homomorphisms
\[
pr_2^{(k)}:E_2^{*,2k}=R \{z_2^k, x_1z_2^k, y_1z_2^k, x_1y_1z_2^k\} \to R\{ x_1z_2^k, y_1z_2^k, x_1y_1z_2^k \}
\]
sending $z_2^k, x_1z_2^k, y_1z_2^k, x_1y_1z_2^k$ to $0, x_1z_2^k, y_1z_2^k, x_1y_1z_2^k$, respectively.
Recall that
\[
d_2(z_1)=x_1y_1-\iota^*(u_2).
\]


The $E_2$-term $E_2^{0,3}$ is spanned by $z_1z_2$.
It is  clear from $d_2(z_2)=0$ that
\[
d_2(z_1z_2)=d_2(z_1)z_2=(x_1y_1-\iota^*(u_2))z_2\not=0.
\]
Hence the homomorphism 
$d_2:E_2^{0,3}\to E_2^{2,2}$ is injective and we have $E_3^{0,3}=\{ 0\}$.


The $E_2$-term $E_2^{2,1}$ is spanned by
\[
m^{(i)}z_1, \; \iota^*(u_2) z_1, \;  x_2 z_1, \; y_2z_1;  \; \; \; v_1 x_1z_1, \; v_1y_1z_1, \; w_1x_1z_1, \; w_1y_1z_1, \;  x_1y_1z_1
\] 
and \[
d_2(\alpha_2 z_1)=\alpha_2 d_2(z_1) =\alpha_2 x_1y_1-\alpha_2 \iota^*(u_2).
\]
for degree $2$ element $\alpha_2$ in $E_2^{2,0}=H^2(BA_2\times B\tilde{\pi}(H_2))$ since $d_2(\alpha_2)=0$.
If $\alpha_2$ is one of $m^{(i)}, \iota^*(u_2), x_2, y_2$, then $\alpha_2 \iota^*(u_2) \in R\{ 1\}$ and so
$pr_2^{(0)}(\alpha_2 \iota^*(u_2))=0$ by definition. Hence, for $\alpha_2=m^{(i)}, \iota^*(u_2), x_2, y_2$, we have
\[
pr_2^{(0)}(d_2(\alpha_2 z_1))=\alpha_2 x_1y_1.
\]
So, we have 
\[
\begin{array}{lcllclclcl}
pr_2^{(0)}(d_2(m^{(i)} z_1))&=& m^{(i)} x_1y_1, \\
pr_2^{(0)}(d_2(\iota^*(u_2)z_1))&=& \iota^{*}(u_2) x_1y_1, \\
pr_2^{(0)}(d_2(x_2z_1))&=&  x_2 x_1y_1, \\
pr_2^{(0)}(d_2(y_2z_1))&=&  y_2x_1y_1. \\
\end{array}
\]
If $\alpha_2$ is one of $v_1x_1, v_1y_1, w_1x_1, w_1y_1, x_1y_1$, then $\alpha_2 x_1y_1=0$. So, 
we have
\[
d_2(\alpha_2 z_1)=-\alpha_2\iota^*(u_2)=-\iota^{*}(u_2)\alpha_2.
\]
By Proposition~\ref{proposition:5.2}, $\iota^*(u_2)w_1 = 0$ in $H^*(B\tilde{\pi}(H_2))$.
Using this, we have 
\[
\begin{array}{lcllclclcl}
d_2(w_1x_1z_1)&=&   -\iota^*(u_2) w_1x_1&=& 0,   \\
d_2(w_1y_1z_1)&=&  -\iota^*(u_2) w_1y_1&=& 0.
\end{array}
\]
Also,
we have
\[
\begin{array}{lcllclclcl}
pr_2^{(0)}(d_2(v_1x_1z_1))&=& &  -  \iota^*(u_2) v_1x_1, \\
pr_2^{(0)}(d_2(v_1y_1z_1))&=& &&  -  \iota^*(u_2) v_1y_1, \\
pr_2^{(0)}(d_2(x_1y_1z_1))&=&  -\iota^*(u_2) x_1y_1.
\end{array}
\]
By Proposition~\ref{proposition:5.1}, $\iota^{*}(u_2)v_1\not=0$.  
So, the kernel of $pr_2^{(0)} \circ d_2$ is spanned by 
\[
x_1y_1z_1+\iota^*(u_2)z_1, \; w_1x_1z_1, \;
w_1y_1z_1.
\]
On the other hand, we have
\[
d_2(x_1y_1z_1+\iota^*(u_2)z_1)=x_1y_1 (x_1y_1-\iota^*(u_2))+ \iota^*(u_2)( x_1y_1 -\iota^*(u_2))=-\iota^*(u_2)^2, 
\]
and since $\iota^*(u_2)^2\not=0$ by Proposition~\ref{proposition:5.1},
$x_1y_1z_1+\iota^*(u_2)z_1$ is not in the kernel of $d_2$.
Hence, the kernel of $d_2$ is spanned by $w_1x_1z_1$, $w_1y_1z_1$
and
the image of $d_2:E_2^{0,2}\to E_2^{2,1}$ is trivial since $E_2^{0,2}$ is spanned by $z_2$ and $d_2(z_2)=0$.
Thus, we have $E_3^{2, 1}=\mathbb{Z}/p\{ w_1x_1z_1, w_1y_1z_1\}$.


As for the $E_2$-term $E_2^{1,3}$, we have a basis 
\[
\{\;  x_1z_1z_2,\; y_1z_1z_2, \;v_1z_1z_2, \; w_1z_1z_2\; \}
\]
 and
 \[
 d_2(\alpha_1 z_1z_2)=-\alpha_1 d_2(z_1)z_2=-\alpha_1 x_1y_1 z_2+\alpha_1 \iota^{*}(u_2) z_2
 \]
for $\alpha_1=x_1, y_1, v_1, w_1$ since $d_2(\alpha_1)=d_2(z_2)=0$.
For $\alpha_1=x_1, y_1$, since $\alpha_1 x_1y_1=0$, we have
\[ 
d_2(\alpha_1z_1z_2)=\alpha_1 \iota^*(u_2)z_2=\iota^*(u_2) \alpha_1 z_2.
\]
For $\alpha_1=v_1, w_1$, since $\alpha_1 \iota^*(u_2)z_2\in R\{z_2\}$, $pr_2^{(1)}(\alpha_1 \iota^*(u_2)z_2)=0$ by definition. Hence, we have
\[ 
pr_2^{(1)}(d_2(\alpha_1z_1z_2))=-\alpha_1 x_1y_1z_2.
\]
Thus, we obtain
\[
\begin{array}{lcllll}
pr_2^{(1)}(d_2(x_1z_1z_2))&=&&\iota^*(u_2)x_1 z_2, \\
pr_2^{(1)}(d_2(y_1z_1z_2))&=&&&\iota^*(u_2)y_1 z_2, \\
pr_2^{(1)}(d_2(v_1z_1z_2))&=&-v_1 x_1y_1z_2, \\
pr_2^{(1)}(d_2(w_1z_1z_2))&=&-w_1 x_1y_1 z_2. &&&
\end{array}
\]
Hence, it is clear that the composition 
\[
pr_2^{(1)} \circ d_2:E_2^{1,3}\to E_2^{3,2}\to R\{ x_1z_2, y_1z_2, x_1y_1z_2\}
\] is injective and so is $d_2:E_2^{1,3}\to E_2^{3,2}$.
Therefore, 
we have $E_{3}^{1,3}=\{ 0\}$.


Next, we compute the $E_4$-terms  $E_4^{0,4}$, $E_4^{1,2}$, $E_4^{2,2}$.
In the $E_3$-term,  the relations are given by $x_1y_1= \iota^*(u_2)$, $\iota^*(u_2)x_1=0$, 
$\iota^*(u_2)y_1= 0$. In particular, $\iota^{*}(u_2)^2=0$.
For simplicity, we write $R'$, $R''$ for $R/(\iota^*(u_2))$, $R/(\iota^*(u_2)^2)$, respectively. 
We have 
\[
E_3^{*, 2k}=R' \{ x_1z_2^k, y_1z_2^k\} \oplus R''\{ z_2^k \},
 \]
as a graded $\mathbb{Z}/p$-module.
Let $N$ be the subspace of $R'\{ x_1\}$ 
spanned by elements of the form $x x_1$, where $x$ ranges 
over a  basis for $H^{*}(B\tilde{\pi}(H_2))/(\iota^*(u_2))\subset R'$. 
Here, we emphasize  that $N$ is not an $R$-submodule
and  that $\tilde{x} m^{(i)} x_1,  \tilde{x}x_1, \tilde{x} v_1x_1, \tilde{x} w_1x_1$ are 
linearly independent  in
$R'\{ x_1\}/N$
, where $\tilde{x}$ ranges over positive degree monomials in 
$x_2$, $y_2$.
We consider a $\mathbb{Z}/p$-module homomorphism
\[
pr_3:E_3^{*,0}=R'\{  x_1, y_1 \}\oplus R''\{1\} \to R'\{ x_1\}/N \oplus R''\{ 1\},
\]
sending $r' x_1, r'  y_1, r''$ to $r'  x_1, 0, r''$, respectively, where $r' \in R'$ and $r'' \in R''$.
Recall that 
\[
d_3(z_2)=x_2y_1-x_1y_2-\iota^*(u_3).
\]


The $E_3$-term $E_3^{0,4}$ is spanned by $z_2^2$ and 
since  $ y_2 x_1z_2$ is
nontrivial in $R'\{ x_1z_2\}$,
 \[
 d_3(z_2^2)=2 d_3(z_2)z_2=2x_2 y_1 z_2-2x_1y_2 z_2-2\iota^*(u_3)z_2 =
 -2y_2 x_1 z_2 +2x_2 y_1z_2-2\iota^*(u_3)z_2
 \]
is nontrivial  in $E_3^{*,2}= R'\{ x_1z_2, y_1z_2\} \oplus R''\{ z_2\}$.
Hence, $d_3:E_3^{0,4}\to E_3^{3,2}$ is injective and we have $E_4^{0,4}=\{0\}$.


The $E_3$-term $E_3^{1,2}$
is spanned by
\[
 v_1 z_2, \; w_1 z_2, \; x_1z_2, \; y_1 z_2,
\] 
since the subspace of $R''$ spanned by degree $1$ elements is same to $H^{1}(B\tilde{\pi}(H_2))$
and since $H^{1}(B\tilde{\pi}(H_2))$ is spanned by $v_1, w_1$.
For $\alpha_1 = v_1, w_1, x_1, y_1$, since $d_3(\alpha_1)=0$, we have 
\[
d_3(\alpha_1z_2)=-\alpha_1 d_3(z_2)=-\alpha_1 x_2 y_1+\alpha_1 x_1y_2 +\alpha_1 \iota^*(u_3).
\]
Hence, for $\alpha_1=v_1, w_1$, since $pr_3(\alpha_1 x_2 y_1)=0$ by definition, we have
\[
pr_3(d_3(\alpha_1z_2))=\alpha_1 x_1 y_2 +\alpha_1\iota^*(u_3)= y_2 \alpha_1  x_1+ \alpha_1 \iota^*(u_3).
\]
For $\alpha_1=x_1, y_1$, since $x_1^2=y_1^2=0$, $x_1y_1=\iota^*(u_2)$, $y_1x_1=-\iota^{*}(u_2)$, we have
\[
d_3(x_1z_2)=-x_1 x_2 y_1+x_1\iota^*(u_3)=-\iota^*(u_3)x_1-x_2  \iota^{*}(u_2) 
\]
and 
\[
d_3(y_1z_2))=y_1 x_1y_2+y_1\iota^*(u_3)=-\iota^*(u_3) y_1-y_2 \iota^*(u_2).
\]
Since $\iota^*(u_3)x_1$ is in $N$, $pr_3(\iota^{*}(u_3)x_1)=0$.
By definition, $pr_3( \iota^{*}(u_3)y_1)=0$.
Therefore, we have 
\[
\begin{array}{lclclcl}
pr_3(d_3(v_1z_2))&  = &v_1y_2x_1  & +& v_1 \iota^*(u_3), \\
pr_3( d_3(w_1z_2))& = &w_1y_2x_1 & +& w_1 \iota^*(u_3), \\
pr_3(d_3(x_1z_2))&= & & - & x_2\iota^{*}(u_2) ,\\
pr_3(d_3(y_1z_2))& =& & -  &y_2\iota^{*}(u_2) .
\end{array}
\]
Since $v_1y_2x_1$, $w_1y_2x_1$ are linearly independent in $R'\{ x_1\}/N$ and since 
$\iota^*(u_2)x_2$, $\iota^{*}(u_2)y_2$ are linearly independent in 
the subspace $\mathbb{Z}/p\{ x_2, y_2\} \otimes H^{2}(B\tilde{\pi}(H_2))
\subset R''\{ 1\}$, 
the four elements 
\[
d_3(v_1z_2), \; d_3(w_1z_2), \; d_3(x_1z_2), \; d_3(y_1z_2)
\]
are linearly independent in $E_3^{*,0}= R'\{ x_1, y_1\}\oplus R''\{ 1\}$.
Hence, the homomorphism $d_3:E_3^{1,2}\to E_3^{4,0}$ 
is injective.   
Therefore, we have $E_4^{1,2}=\{0\}$.


The $E_3$-term $E_3^{2,2}$ is spanned by 
 \[
 m^{(i)} z_2, \; 
     \iota^*(u_2) z_2, \; 
   x_2 z_2, \;
   y_2z_2; \;
 v_1x_1z_2, \; 
 v_1y_1z_2, \;
 w_1x_1z_2, \;
w_1y_1z_2.
\]
For $\alpha_2= m^{(i)},  \iota^*(u_2), x_2, y_2, v_1x_1, v_1y_1, w_1x_1,w_1 y_1\in E_3^{2,0}$,
since $d_3(\alpha_2) \in E_3^{5,-2}=\{0\}$, 
we have \[
d_3( \alpha_2 z_2)=\alpha_2 d_3(z_2)=\alpha_2 x_2 y_1- \alpha_2 x_1y_2-\alpha_2 \iota^*(u_3).
\]
For $\alpha_2=m^{(i)},  \iota^*(u_2), x_2, y_2$, since $pr_3(\alpha_2 x_2y_1)=0$ by definition,
we have
\[
pr_3(d_3( \alpha_2 z_2))=-\alpha_2 y_2 x_1 - \alpha_2 \iota^{*}(u_3).
\]
Thus, we have 
\[
\begin{array}{lcllll}
pr_3(d_3( m^{(i)} z_2)) &=& -  y_2 m^{(i)}  x_1    -  m^{(i)} \iota^*(u_3), \\
pr_3(d_3( x_2 z_2)) &=& -  x_2 y_2 x_1    -  x_2 \iota^*(u_3), \\
pr_3(d_3( y_2 z_2)) &=& -  y_2^2 x_1    -  y_2 \iota^*(u_3).
\end{array}
\]
Moreover, since $\iota^{*}(u_2)\iota^*(u_3)=\iota^{*}(u_2u_3)=0$ in $H^{*}(B\tilde{\pi}(H_2))$ by Proposition~\ref{proposition:3.2}, 
and since
$\iota^*(u_2)x_1=\iota^*(u_2)y_1=0$ in $R' \{ x_1, y_1\}$,  
we have 
\[
d_3(\iota^{*}(u_2)z_2)=0.
\]
For $\alpha_1=v_1, w_1$, using the relations $x_1^2=y_1^2=0$, $x_1y_1=\iota^*(u_2)$, $y_1x_1=-\iota^*(u_2)$ in $E_3$, we have
\[
d_3( \alpha_1 x_1z_2)=\alpha_1 x_1 x_2 y_1-\alpha_1 x_1 x_1 y_2 - \alpha_1 x_1\iota^*(u_3)= \alpha_1 \iota^*(u_3) x_1+x_2 \alpha_1 \iota^*(u_2)
\]
and
\[
d_3( \alpha_1 y_1z_2)=\alpha_1 y_1 x_2 y_1-\alpha_1 y_1 x_1 y_2 - \alpha_1 y_1\iota^*(u_3)=
\alpha_1 \iota^*(u_3) y_1+ y_2 \alpha_1 \iota^*(u_2).
\]
Since $\alpha_1 \iota^*(u_3) \in H^{*}(B\tilde{\pi}(H_2))/(\iota^*(u_2))$, we obtain 
$\alpha_1 \iota^*(u_3) x_1\equiv 0$ in $R'\{x_1\}/N$, hence $pr_3(\alpha_1 \iota^*(u_3)x_1)=0$. Moreover, $pr_3(\alpha_1 \iota^*(u_3) y_1)=0$ by definition. 
So, we have
\[
pr_3(d_3(\alpha_1 x_1z_2))=\alpha_1x_2 \iota^*(u_2)=x_2\alpha_1 \iota^*(u_2)
\]
and
\[
pr_3(d_3(\alpha_1 y_1z_2))=\alpha_1y_2 \iota^*(u_2)=y_2\alpha_1 \iota^*(u_2).
\]
By Proposition~\ref{proposition:5.2}, $w_1 \iota^{*}(u_2)=0$. Hence, we have
 \[
  d_3(w_1x_1z_2)=w_1\iota^{*}(u_3)x_1
  \]
  and 
  \[
  d_3(w_1y_1z_2)=w_1\iota^{*}(u_3)y_1.
  \]
Furthermore, by Proposition~\ref{proposition:5.2}, 
 $Q_0 (w_1\iota^*(u_2))=Q_0 w_1 \cdot  \iota^{*}(u_2)-w_1\iota^{*}(u_3)=0$ in $H^{*}(B\tilde{\pi}(H_2))$,  
hence  $w_1\iota^{*}(u_3)x_1=(Q_0 w_1) \iota^*(u_2) x_1 = 0$ in $R'\{x_1, y_1\}\subset E_3^{*,0}$. 
Thus, we obtain $d_3(w_1x_1z_2)=0$.
Similarly, we also have 
$
  d_3(w_1y_1z_2)=0.
$
Thus, we have
\[
\begin{array}{lcllll}
pr_3(d_3( v_1x_1 z_2)) &=&      x_2 v_1 \iota^*(u_2)  , \\
pr_3(d_3( v_1y_1 z_2)) &=&      y_2 v_1 \iota^*(u_2), 
\end{array}
\]
and
\[
\begin{array}{lcllll}
d_3( w_1x_1 z_2)&=& 0, \\
d_3( w_1y_1 z_2)&=& 0.
\end{array}
\]
Since $y_2 m^{(i)} x_1, x_2y_2 x_1,  y_2^2 x_1$ are linearly independent 
in $R' \{ x_1\}/N$ and since, by Proposition~\ref{proposition:5.1}, 
$x_2 v_1\iota^*(u_2) $, $y_2 v_1 \iota^*(u_2) $ are linearly independent in $\mathbb{Z}/p\{ x_2, y_2\}
\otimes H^{3}(B \tilde{\pi}(H_2))\subset R''\{1\}$, 
the kernel of $pr_3\circ d_3$ is spanned by 
$\iota^*(u_2)z_2, w_1x_1z_2, w_1y_1z_2$
and since these are in the kernel of $d_3$, the kernel of $d_3$ is spanned by these elements.
Moreover, the image $d_3:E_3^{-1,4}\to E_3^{2,2}$ is trivial. Therefore,  we 
obtain
\[
E_4^{2,2}=\mathbb{Z}/p\{ \iota^{*}(u_2)z_2, w_1x_1z_2, w_1y_1z_2\}=\mathbb{Z}/p\{ x_1y_1z_2, w_1x_1z_2, w_1y_1z_2\}, 
\]
where $\iota^*(u_2)z_2=x_1y_1z_2$.


Finally, we compute $E_\infty$-terms $E_\infty^{0,3}, E_\infty^{1,2}, E_\infty^{2,1}$ and 
$E_\infty^{0,4}, E_\infty^{1,3}, E_\infty^{2,2}$. 
Since $E_3^{0,3}=E_4^{1,2}=\{0\}$, we have $E_\infty^{0,3}=E_\infty^{1,2}=\{0\}$.
Similarly, since $E_4^{0,4}=E_3^{1,3}=\{0\}$, we have $E_\infty^{0,4}=E_\infty^{1,3}=\{0\}$.
Since the Leray-Serre spectral sequence is the first quadrant spectral sequence, 
for $s\leq r-1$, $t\leq r-2$, 
\[
E_r^{s-r,t+r-1}=E_r^{s+r, t-r+1}=\{0\},
\]
and differentials
\[
d_r:E_r^{s-r,t+r-1} \to E_r^{s,t}, \quad d_r:E_r^{s,t}\to E_r^{s+r,t-r+1}
\]
 are trivial. 
Hence, we have $E_r^{s, t}=E_\infty^{s,t}$ for $s\leq r-1$, $t\leq r-2$. In particular, 
$E_3^{s,t}=E_\infty^{s,t}$ for $s\leq 2$, $t\leq 1$,
$E_4^{s,t}=E_\infty^{s,t}$ for $s\leq 3$, $t\leq 2$.
Hence, we have  $E_\infty^{2,1}=E_3^{2,1}$, $E_\infty^{2,2}=E_4^{2,2}$.
\end{proof}


In Section~\ref{section:4}, we defined $x_4\in H^4(BG;\mathbb{Z})$, so that $c_2(\lambda'')=p x_4$ in $H^4(BG;\mathbb{Z})$. Therefore, to show that $c_2(\lambda''\circ f) \not=0$ in $H^4(BH;\mathbb{Z})$
is equivalent to show that
 $p f^*(x_4)\not=0$ in $H^4(BH;\mathbb{Z})$.
Hence, in order to prove (1), it suffices to show that 
the mod $p$ reduction $\rho(f^*(x_4))\in H^4(BH)$ of $f^*(x_4)\in H^4(BH;\mathbb{Z})$ is 
not in the image of the Bockstein homomorphism. So, we prove the following 
 proposition.
 
\begin{proposition}
\label{proposition:6.2}
The cohomology class
$f^*(\rho(x_4))$ is not in the image of the Bockstein homomorphism
\[
Q_0:H^3(BH) \to H^4(BH).
\]
\end{proposition}
 \begin{proof}
 Since $E_\infty^{0,4}=E_\infty^{1,3}=\{ 0\}$, we have $F^2H^4(BH)=H^4(BH)$.
 Similarly, since $E_\infty^{0,3}=E_\infty^{1,2}=\{ 0\}$, 
 we have $F^2H^3(BH)=H^3(BH)$.
Hence, we have 
\[
Q_0H^3(BH)\subset F^2H^4(BH)
\]
and each cohomology class in $Q_0H^3(BH)$ 
represents an element in 
\[
E_{\infty}^{2,2}=F^2H^4(BH)/F^3H^4(BH).
\]
Since $E_\infty^{2,1}$ is spanned by $w_1x_1z_1$, $w_1y_1z_1$, 
using the properties of the vertical operation $\beta \wp^0$ constructed by Araki in the
spectral sequence of a fibration \cite[Corollary 4.1]{araki-1957}, we have that 
if $x$ is in $Q_0H^3(BH)$, then $x$ represents 
a linear combination of $w_1x_1z_2$, $w_1y_1z_2$
in $E_{\infty}^{2,2}$.

On the other hand, by Proposition~\ref{proposition:4.2}, 
$\rho(x_4) \in H^4(BG)$ represents   $\alpha b_2 z_2\in E_\infty^{2,2}(BG)$, where $\alpha \not = 0 \in \mathbb{Z}/p$.
Using Proposition~\ref{proposition:3.3} and the definition of $b_2$ in Section~\ref{section:4}), we have
$f^*(b_2)=x_1y_1$.
Therefore,
$f^*(\rho(x_4))$ represents 
$\alpha x_1y_1z_2$ in $E_\infty^{2,2}$. Hence, $f^*(\rho(x_4))$
is not in the image of the 
Bockstein homomorphism $Q_0$.
\end{proof}

\begin{remark}\label{remark:6.3}
If we replace $H_2$ by the extraspecial $p$-group $p_{+}^{1+2}$, then (1) does not hold.
To be more precise, $f^{*}(\rho(x_4))$ is in the image of the Bockstein homomorphism
$Q_0:H^3(Bp_{+}^{1+4})\to H^4(Bp_{+}^{1+4})$ and $c_2(\lambda''\circ f) =p f^*(x_4)=0$ 
in $H^4(Bp_{+}^{1+4};\mathbb{Z})$.
\end{remark}

Finally, we prove  (2) by proving the following proposition.
\begin{proposition}
\label{proposition:6.4}
There exists no virtual complex representation 
\[
\mu:BH\to \mathbb{Z}\times BU
\]
 such  that 
$c_2(\lambda''\circ f)\in p \cdot \mathrm{Im}\, {\mu}^*$.
\end{proposition}

\begin{proof}
Proof by contradiction. 
Suppose that there exists a virtual complex representation 
\[
\mu:BH\to \mathbb{Z}\times BU
\]
 such that 
$c_2(\lambda''\circ f)\in p \cdot \mathrm{Im} {\mu}^*$.
Then, $p({\mu}^*(y_4)-f^*(x_4))=0$ for some 
$y_4\in H^4(\mathbb{Z} \times BU;\mathbb{Z})$.
Since  $Q_1$ acts trivially on $H^{*}(\mathbb{Z} \times BU)$, we have
\[
Q_1\rho(\mu^*(y_4))=0.
\]
In what follows, we show that
\[
Q_1\rho(\mu^*(y_4))\not =0,
\]
which proves the proposition.

Since $p({\mu}^*(y_4)-f^*(x_4))=0$, $\rho({\mu}^*(y_4)-f^*(x_4))$ is in the image of the 
Bockstein homomorphism, that is, as in the proof of Proposition~\ref{proposition:6.2}, 
$\rho({\mu}^*(y_4)-f^*(x_4))$ represents 
\[
\alpha_1 w_1x_1z_2+\alpha_2 w_1y_1z_2
\]
 in $E_\infty^{2,2}$ for some $\alpha_1, \alpha_2 \in \mathbb{Z}/p$.
We already verified that $f^*(\rho(x_4))=\rho(f^*(x_4))$
 represents 
 $\alpha x_1y_1z_2\in E_\infty^{2,2}$, where $\alpha\not=0$,
 in the proof of Proposition~\ref{proposition:6.2}. So,
$\rho({\mu}^*(y_4))$ represents
\[
\alpha x_1y_1z_2+\alpha_1 w_1x_1z_2+\alpha_2 w_1y_1z_2
\]
 in $E_\infty^{2,2}$ and $\alpha\not=0$.
 
 We recall the structure of $H_2$ defined in Section~\ref{section:2}. Also, we recall the diagram
 \[
\begin{diagram}
\node{A_3} \arrow{e,t}{g}  \arrow{s,l}{\varphi} 
\node{H}  \arrow{s,r}{\pi}
\node{A_3'} \arrow{s,r}{\varphi'}\arrow{w,t}{g'}
\\
\node{A_2} \arrow{e,t}{g}
\node{A_2\times \tilde{\pi}(H_2)} 
\node{A_2.}\arrow{w,t}{g'}
\end{diagram}
\]
where  upper $g, g'$ are obvious inclusions, 
$A_2=\langle \tilde{\pi}(\alpha), \tilde{\pi}(\beta) \rangle$, 
\[
\begin{array}{ll}
g(\tilde{\pi}(\alpha))=(\tilde{\pi}(\alpha), \tilde{\pi}(\alpha)), &g(\tilde{\pi}
(\beta))=(\tilde{\pi}(\beta), \tilde{\pi}(\beta)), 
\\[.2cm]
g'(\tilde{\pi}(\alpha))=(\tilde{\pi}(\alpha), 1), & g'(\tilde{\pi}(\beta))=(1, \tilde{\pi}(\beta)).
\end{array}
\]
In Section~\ref{section:5}, we defined $w_1\in H^1(B\tilde{\pi}(H_2))$, so that the induced homomorphism
$H^1(B\tilde{\pi}(H_2))\to H^1(B\langle \tilde{\pi}(\beta) \rangle)$ maps $w_1$ to the element corresponding to the generator $\tilde{\pi}(\beta)$.
So, 
we see that  the induced homomorphisms $g^*$, ${g'}^*$ satisfy
\[
\begin{array}{llll}
g^*(x_1) = x_1, & g^*(y_1)=y_1, & g^*(w_1) =y_1,
 \\
 {g'}^*(x_1)=x_1, 
&
 {g'}^*(y_1)=0,
&
 {g'}^*(w_1)=y_1.
\end{array}
\]
Therefore, $g^*(\rho(\mu^*(y_4)))\in H^4(BA_3)$ represents 
\[
g^*(\alpha x_1y_1z_2+\alpha_1 w_1x_1z_2+\alpha_2 w_1y_1z_2)=\alpha x_1y_1z_2+\alpha_1 y_1x_1z_2=(\alpha-\alpha_1) x_1y_1z_2
\]
in  the spectral sequence for $H^{*}(BA_3)$
and ${g'}^*(\rho(\mu^*(y_4)))\in H^4(BA_3')$ represents
\[
{g'}^*(\alpha x_1y_1z_2+\alpha_1 w_1x_1z_2+\alpha_2 w_1y_1z_2)=\alpha_1 y_1x_1z_2=-\alpha_1 x_1y_1z_2 
\]
in  the spectral sequence for $H^{*}(BA_3')$.

As in the proof of Proposition~\ref{proposition:4.1}, 
let $M, M'$ be the $\varphi^{*}(H^*(BA_2))$-submodule of $H^{*}(BA_3)$, the ${\varphi'}^*(H^{*}(BA_2))$-submodule 
of $H^{*}(BA_3')$, generated by
\[
1, \;z_2^{i}, \; z_1, \; z_1z_2,  \; z_1z_2^i \; (i\geq 2), 
\]
where $\varphi:BA_3\to BA_1$, $\varphi':BA_3'\to BA_2$ are the maps defined in Section~\ref{section:2}, so that 
\[
H^{*}(BA_3)/M=\varphi^{*}(H^{*}(BA_2)) \{ z_2 \}=\mathbb{Z}/p[x_2, y_2]\otimes \Lambda(x_1, y_1) \{z_2\}, 
\]
and
\[
H^{*}(BA'_3)/M'={\varphi'}^{*}(H^{*}(BA_2)) \{ z_2 \}=\mathbb{Z}/p[x_2, y_2]\otimes \Lambda(x_1, y_1) \{z_2\},
\]
respectively.
Since $Q_1 z_1=z_2^p$, $Q_1z_2=0$, and $Q_1$ is a derivation, $M, M'$ are closed under the action of 
Milnor operation $Q_1$.
Modulo $M, M'$, 
we have 
\begin{align*}
g^*(\rho(\mu^*(y_4)))
&
\equiv(\alpha-\alpha_1)x_1y_1z_2 \mod M,
\\
{g'}^*(\rho(\mu^*(y_4)))
&
\equiv
-\alpha_1x_1y_1z_2  \mod M'.
\end{align*}
and so 
\begin{align*}
Q_1g^*(\rho(\mu^*(y_4)))&
\equiv (\alpha-\alpha_1) (x_2^py_1 - x_1y_2^p)z_2  \mod M, 
\\
Q_1{g'}^*(\rho(\mu^*(y_4)))
&
\equiv
-\alpha_1 (x_2^py_1- x_1y_2^p)z_2 \mod M'.
\end{align*}
Since $\alpha\not=0$, at least one of $\alpha-\alpha_1$, $-\alpha_1$ is non-zero.
Therefore, we have
\[
Q_1\rho(\mu^*(y_4))\not=0.
\]
This completes the proof.
\end{proof}


\begin{thebibliography}{9}

\bibitem{antieau-williams-2014}
B. Antieau\ and\ B. Williams, 
The topological period-index problem over 6-complexes,
Journal of Topology {\bf 7} (2014), 617--640.

\bibitem{araki-1957}
S. Araki, Steenrod reduced powers in the spectral sequences associated with a fibering. II, 
Mem. Fac. Sci. Kyusyu Univ. Ser. A. Math. {\bf 11} (1957), 81--97. 

\bibitem{baum-browder-1965}
P. F. Baum\ and\ W. Browder, The cohomology of quotients of classical groups, 
Topology {\bf 3} (1965), 305--336.

\bibitem{kameko-yagita-2008}
M. Kameko\ and\ N. Yagita, The Brown-Peterson cohomology of the classifying spaces
 of the projective unitary groups ${\rm PU}(p)$ and exceptional Lie groups, 
 Trans. Amer. Math. Soc. {\bf 360} (2008), no.~5, 2265--2284. 


\bibitem{kameko-2015}
M. Kameko, On the integral Tate conjecture over finite fields, Math. Proc. Cambridge Philos. Soc. {\bf 158} (2015), no.~3, 531--546. 

\bibitem{kono-mimura-shimada-1975}
A. Kono, M. Mimura\ and\ N. Shimada, 
Cohomology of classifying spaces of certain associative $H$-spaces, 
J. Math. Kyoto Univ. {\bf 15} (1975), no.~3, 607--617. 

\bibitem{totaro-1999}
B. Totaro, The Chow ring of a classifying space, in {\it Algebraic $K$-theory
 (Seattle, WA, 1997)}, 249--281, Proc. Sympos. Pure Math., 67, Amer. Math. Soc., 
 Providence, RI. 

\bibitem{totaro-2014}
B. Totaro,  {\it Group cohomology and algebraic cycles},
Cambridge Tracts in Mathematics, 204, Cambridge Univ. Press, 2014.

\bibitem{vistoli-2007}
A. Vistoli, On the cohomology and the Chow ring of the classifying space of 
${\rm PGL}\sb p$, J. Reine Angew. Math. {\bf 610} (2007), 181--227. 


\end{thebibliography}
 \end{document}